\documentclass{svjour3}
\usepackage{graphicx}
\usepackage{mathptmx}      % use Times fonts if available on your TeX system
\usepackage{amsfonts}
\usepackage[english]{babel}
\usepackage{amsmath}
\usepackage{enumerate}% http://ctan.org/pkg/enumerate
\usepackage{epstopdf} %converting to PDF
\usepackage{multirow}
\usepackage{algorithm}
\usepackage{algorithmic}
\usepackage{circuitikz}
\usepackage{subcaption}
\usepackage[utf8]{inputenc}
\usepackage{algorithm}
\usepackage{algorithmic}
\usepackage{tikz}
\usetikzlibrary{shapes.geometric, arrows}
\usetikzlibrary{patterns}
\usetikzlibrary{automata,positioning}

\newtheorem{thm}{Theorem}[section]

\newtheorem{lem}[thm]{Lemma}

\newtheorem{rem}[thm]{Remark}
\newtheorem{exa}[thm]{Example}
\numberwithin{equation}{section} \numberwithin{thm}{section}

\newcommand{\norm}[1]{\left\Vert#1\right\Vert}

\newcommand{\be}{\begin{equation}}
\newcommand{\ee}{\end{equation}}
\newcommand{\bea}{\begin{eqnarray}}
\newcommand{\eea}{\end{eqnarray}}

\newcommand{\sign}{\text{sign}}

\begin{document}
\sloppy
\title{On the Zeros of $q$-Hankel Transform by Using P\'{o}lya-Hurwitz Partial Fraction Method}



%\author[MathSciCai]{M.~H.~Annaby \corref{cor1}}
%\ead{mhannaby@sci.cu.edu.eg}
%\author[MathFym]{S.R. El-Sayed}
%\address[MathSciCai]{Department of Mathematics, Faculty of Science,
%Cairo University, Giza 12613, Egypt}
%\address[MathFym]{Department of Mathematics, Faculty of Science, Fayoum University, Fayoum, Egypt}
\author{M.~H.~Annaby         \and
        S.~R.~Elsayed-Abdullah
}

\institute{M.~H.~Annaby \at
              Department of Mathematics, Faculty of Science, Cairo University, 12613 Giza, Egypt \\
              %Tel.: +20-10-14455876\\
              %Fax: +20-23-572884\\
              \email{mhannaby@sci.cu.edu.eg}           %  \\
           \and
           S.~R.~Elsayed-Abdullah \at
           Department of Mathematics, Faculty of Science, Fayoum University, Fayoum, Egypt \\
           \email{sre12@fayoum.edu.eg}
}

\date{Received: date / Accepted: date}
\maketitle

\begin{abstract}
The technique of P\'{o}lya-Hurwitz of partial fractions is implemented to investigate the zeros of finite $q$-Hankel transforms, which are defined in terms of the third $q$-Bessel function of Jackson. The new approach, which is a $q$-counterpart of P\'{o}lya-Hurwitz technique
relaxes the restrictive conditions imposed on
$q$ in the previously obtained results. In the present study, we use the
$q$-type sampling theorems of the $q$-Hankel transforms, which lead directly to $q$-partial fractions. Various experimental examples are established.

\end{abstract}
\keywords{Zeros of entire functions, P\'{o}lya-Hurwitz partial fraction expansions, $q$-sampling theory\\
	2020 MSC. 30D10, 30C15, 30D45, 39A13.}

%\keywords{}

%\end{frontmatter}
%%%%%%%%%%%%%%%%%%%%%%%%%%%%%

\section{Introduction}

%%%%%%%%%%%%%%%%%%%%%%%%%%%%%%%

Throughout this paper $0<q<1$ is a fixed number. Consider the $q$-Hankel (basic Hankel) transform
\begin{equation}\label{eq:1.1}
H_q^\nu(f;z)=\int_{0}^{1}f(t)\left(tz\right)^{1/2}J_\nu(tz;q^2)\,d_qt, 
\end{equation}
where $J_v(\cdot;q^2)$ is the third Jackson $q$-Bessel function of order
$v>-1$. In this paper we aim
to investigate the distribution of zeros of (\ref{eq:1.1}) by extending the results of \cite{AnSh1} for the basic cosine
and sine transforms to the $q$-Hankel transform (\ref{eq:1.1}). We give a $q$-analog of P\'{o}lya-Hurwitz
results without adding further restrictions on $f$ rather than the standard ones, and it turns out that this can be done by a
$q$-analog of Hurwitz partial fraction technique. For this task, we will use the
$q$-sampling theorem obtained by Annaby et al.\cite{AnHM2012}. In addition, we will also use another $q$-sampling theorem \cite{AnHM2012} to investigate the distribution of zeros of the function $F_q^\nu(f;z)$, defined below for $\nu \in \left(0,1\right)-\{1/2\}$.

In this section we present the Hahn-Exton $q$-Bessel function and some of its properties.
Here $q$ $\in$ $(0,1)$ is fixed, the Hahn-Exton $q$-Bessel function $J_\nu(z;q)$ of order $\nu>-1$ is defined
to be
\begin{equation}\label{eq:1.2}
J_\nu(z;q) := z^\nu \frac{\left(q^{\nu+1};q\right)_\infty}{\left(q;q\right)_\infty} \; {}_1\phi_1\left(\begin{array}{c} 0\\ q^{\nu+1}\end{array}\\|
q;qz^2 \right),
\quad z \in \mathbb{C},
\end{equation}
where the $q$-hypergeometric series ${}_r\phi_s$ is defined via
\begin{equation}\label{eq:1.3}
{}_r\phi_s \left( \begin{matrix} a_1,\dots,a_r \\ b_1,\dots,b_s \end{matrix}| q; z \right)
= \sum_{k=0}^\infty \frac{(a_1;q)_k(a_2;q)_k \dots (a_r;q)_k}{(b_1;q)_k(b_2;q)_k \dots (b_s;q)_k} 
\left((-1)^k q^{\frac{1}{2}k(k-1)} \right)^{1+s-r} z^k,
\end{equation}
whenever the series converges.
The $q$-shifted factorial, cf \cite{GRR}, is defined for $a \in \Bbb C$ to be
\begin{equation}\label{eq:1.4}
(a;q)_k=
\begin{cases}
1, & k=0,\\
\prod\limits_{i=0}^{k-1}(1-aq^i),  & k=1,2,\cdots,
\end{cases}
\end{equation}
and $(a;q)_\infty:=\lim\limits_{k\longrightarrow \infty}(a;q)_k$ is well defined. The multiple $q$-shifted factorial for complex numbers $a_1,a_2,\cdots,a_n$ is defined to be
\begin{equation}\label{eq:1.5}
(a_1,a_2,\cdots,a_n;q)_k=\prod\limits_{j=1}^{n}(a_j;q)_k.
\end{equation}
In \cite{GRR}, Jackson defined the $q$-integral to be
\begin{equation}\label{eq:1.6}
\int_{a}^{b}f(t)d_qt=\int_{0}^{b}f(t)d_qt-\int_{0}^{a}f(t)d_qt,
\end{equation}
where
\begin{equation}\label{eq:1.7}
\int_{0}^{x}f(t)d_qt:=x(1-q)\sum\limits_{k=0}^{\infty}q^kf(xq^k),
\end{equation}
provided that the series converges. The $q$-difference operator, which generalizes the definition of the classical derivative is defined by Jackson to be
\begin{equation}\label{eq:1.8}
(D_qf)(t)=\frac{f(t)-f(qt)}{(1-q)t}, \ t \neq 0,
\end{equation}
and if $t=0$, $(D_qf)(0)$ is defined to be, see e.g. \cite{AnM12},
\begin{equation}\label{eq:1.9}
(D_qf)(0)=\lim\limits_{n\longrightarrow \infty}\frac{f(tq^n)-f(0)}{tq^n}.
\end{equation}
Under standard conditions $(D_qf)(t)\longrightarrow f'(t)$ as $q \longrightarrow 1$, cf \cite{GRR}, and $q$-Jackson's integral (\ref{eq:1.7}) is a right inverse of the $q$-difference operator (\ref{eq:1.8}).

According to \cite{HR}, the Hahn-Exton $q$-Bessel function $J_\nu(\cdot; q^2)$ has a countably infinite number of positive simple zeros denoted by $j_{k\nu}$, $k =0,1,\cdots,$
where $0 < j_{1\nu}<j_{2\nu}<\cdots$. The asymptotic of $j_{k\nu}$'s is studied in \cite{BuCard2}, under the condition that $q^{2\nu+2}<\left(1-q^2\right)^2$. 

It was proved in \cite{HR} that 
for the non-zero real zeros $j_{k\nu}$, $k =1,2,\cdots,$ of $J_\nu(\cdot;q^2)$, we have
\begin{equation}\label{eq:1.10}
\int\limits_{0}^{1}t\left(J_\nu\left(qj_{k\nu}t;q^2\right)\right)^2\,d_qt=-\frac{1}{2}\left(1-q\right)q^{\nu-1}J_{\nu+1}\left(qj_{k\nu};q^2\right)J_\nu'\left(j_{k\nu};q^2\right).
\end{equation}

In \cite{AnM06} the authors derived results on reality and distribution of the $q$-Hankel transforms
\begin{eqnarray}
\label{eq:1.11}U_{\nu,g}(z)&:=&\int_{0}^{1}\left(tz\right)^{-\nu}g(t)J_\nu(tz;q^2)\,d_qt,\quad z \in \mathbb{C},\\
\label{eq:1.12}V_{\nu,g}(z)&:=&\int_{0}^{1}\left(tz\right)^{1-\nu}g(t)J_\nu(tz;q^2)\,d_qt,\quad z \in \mathbb{C}. 
\end{eqnarray}
Under the condition
\[q^{-1}(1-q)\frac{c_{\nu,g}}{C_{\nu,g}}>1,\]
the authors \cite{AnM06} proved that for $\nu>-1$ and $g(\cdot) \in L_q^1(0,1)$, the zeros of $U_{\nu,g}(z)$ are real, simple
and infinite. Moreover, $U_{\nu,g}(z)$ is an even function with no zeros in $\left[0, q^{-1}/\sqrt{C_{\nu,g}}\right)$, and its
positive zeros lie in the intervals
\begin{equation}\label{eq:1.13}
\left(\frac{q^{-m+1/2}}{\sqrt{C_{\nu,g}}}, \frac{q^{-m-1/2}}{\sqrt{C_{\nu,g}}}\right), \quad m=1,2,\cdots,
\end{equation}
one zero in each interval. In the same manner they proved that under the condition
\[q^{-1}(1-q)\frac{b_{\nu,g}}{B_{\nu,g}}>1,\]
the zeros of $V_{\nu,g}(z)$ are real, simple and infinite. Moreover, the odd function $V_{\nu,g}(z)$ has only one zero $z=0$ in $\left[0, q^{-1}/\sqrt{B_{\nu,g}}\right)$, and its positive zeros lie in the intervals
\begin{equation}\label{eq:1.14}
\left(\frac{q^{-m+1/2}}{\sqrt{B_{\nu,g}}}, \frac{q^{-m-1/2}}{\sqrt{B_{\nu,g}}}\right), \quad m=1,2,\cdots.
\end{equation}
The numbers $C_{\nu,g},\ c_{\nu,g},\ B_{\nu,g}$ and $b_{\nu,g}$ are defined in \cite[p. 1093]{AnM06}. For brevity we don't state these numbers, but we will compute them in some examples in the last section of the paper. These numbers depend on $\nu, q, g(\cdot)$ and their presence in (\ref{eq:1.13}) and (\ref{eq:1.14}) add restrictive conditions on $q, g(\cdot)$, which we aim to improve.

Let $L_q^2(0,1)$ be the set of all $q$-square integrable functions $f(\cdot)$, i.e.,
\begin{equation}\label{eq:1.15}
\norm{f}_2^2=\int_{0}^{1}\left|f(t)\right|^2\,d_qt<\infty.
\end{equation}
The space $L_q^2(0,1)$ is a Banach space with respect to the norm $\norm{\cdot}_2^2$. Moreover, it is a Hilbert space with the inner product 
\begin{equation}\label{eq:1.16}
\left<f,g\right>=\int_{0}^{1}f(t)\overline{g(t)}\,d_qt,\ f,g \in L_q^2(0,1).
\end{equation}

Now, we introduce the following sampling expansions \cite{AnHM2012} for the transforms $H_q^\nu(f;z)$ and $F_q^\nu(f;z)$, which play a major role in the main results of the paper.

\begin{thm}\label{thm:1.1}
For $f(\cdot) \in L_q^2(0,1)$, the function
\begin{equation}\label{eq:1.17}
H_q^\nu(f;z) = \int_0^1 f(t)\left(tz\right)^{1/2}J_\nu(tz; q^2) \;d_qt
\end{equation}
admits the sampling expansion
\begin{equation}\label{eq:1.18}
H_q^\nu(f;z) = \sum_{k=1}^{\infty} \frac{H_q^\nu(f;qj_{k\nu})}{\alpha_k}\left(z^{-1}q j_{k\nu}\right)^{\frac{1}{2}}J_\nu(q^{-1}z; q^2)
\left( \frac{1}{z - qj_{k\nu}} + \frac{1}{z + qj_{k\nu}} \right),
\end{equation}
where 
\begin{equation}\label{eq:1.19}
\alpha_k = q^{-1} J_\nu'(j_{k\nu}; q^2)
\end{equation}
and $j_{k\nu}$ are the positive zeros of $J_\nu(z; q^2)$. The convergence of \textup{(\ref{eq:1.18})} is uniform on any compact subset of $\left(0,\infty\right)$.
\end{thm}
The transform $F_q^\nu\left(f;\cdot\right)$ is defined in equation (\ref{eq:1.21}) below. Let $Y_\nu\left(z;q\right)$ be the function defined for $\nu \notin \Bbb Z$ by
\begin{equation}\label{eq:1.20}
Y_\nu\left(z;q\right)=\frac{\Gamma_q(\nu)\Gamma_q(1-\nu)}{\pi}\left(q^{\nu/2}\cos \pi \nu J_\nu(z;q)-J_{-\nu}(zq^{-\nu/2};q)\right),
\end{equation}
and for $n \in \Bbb Z$, $Y_n\left(z;q\right)=\lim\limits_{\nu \longrightarrow n}Y_\nu\left(z;q\right)$. The following sampling representation is derived in \cite{AnHM2012}. Here we give the exact representation, which corrects the calculations of \cite{AnHM2012}.
\begin{thm}\label{thm:1.2}
For $f(\cdot) \in L_q^2(0,1)$ and $\nu \in \left(0,1\right)-\{1/2\}$, the $q$-Hankel type transform
\begin{equation}\label{eq:1.21}
F_q^\nu\left(f;z\right)= -\frac{q^{-\nu\left(\nu-1\right)-{1/2}}}{1+q}\int_0^1 \overline{f(t)}t^{1/2}\left(J_\nu(tz; q^2)Y_\nu(z; q^2)-Y_\nu(tz; q^2)J_\nu(z; q^2)\right) \;d_qt
\end{equation}
admits the sampling representation
\begin{equation}\label{eq:1.22}
zF_q^\nu\left(f;z\right) = \sum_{k=1}^{\infty}\frac{z_kF_q^\nu\left(f;z_k\right)}{\beta_k}\left(q^{\nu^2}z^{-\nu}J_\nu(z; q^2)-cz^{\nu}J_{-\nu}(q^{-\nu}z; q^2)\right)\frac{2z}{\left(z^2 - z_k^2\right)},
\end{equation}
where 
\begin{equation}\label{eq:1.23}
\beta_k = \frac{d}{dz}\left(q^{\nu^2}z^{-\nu}J_\nu(z; q^2)-cz^{\nu}J_{-\nu}(q^{-\nu}z; q^2)\right)_{z=z_k}.
\end{equation}
For $k=1,2,\cdots,$ $z_k$'s are the positive zeros of $\left(q^{\nu^2}z^{-\nu}J_\nu(z; q^2)-cz^{\nu}J_{-\nu}(q^{-\nu}z; q^2)\right)$ and $c$ is any constant. Moreover the series \textup{(\ref{eq:1.22})} converges absolutely on $\Bbb C$ and uniformly on compact subsets of $\Bbb C$.
\end{thm}

The following classical result of Hurwitz \cite{H1,H2}, see also \cite{Sch,Titch} will be needed in the sequel.
\begin{thm}\textup{(Hurwitz)}\label{thm:1.3}
Let $f_n(z)$ be a sequence of functions defined on a region $\Omega \subseteq \Bbb C$, $f_n$ converges uniformly to $f$ in $\Omega$, $f \neq 0$ on $\Omega$. Let $z_0 \in \Omega$ be an interior point of $\Omega$. Then:
\begin{enumerate}
	\item [\textup{(i)}] $f(z_0)=0$ if and only if $z_0$ is a limit point of the set of all zeros of $f_n(z)$, $n \in \Bbb N$.
	
	\item [\textup{(ii)}] If $z_0$ is a zero of order $m$ of $f(z)$, then there are suitable $\delta>0$, $N_0\in \Bbb N$, such that for all $n \in \Bbb N$, $n>N_0$, $f_n(z)$ has exactly $m$ zeros in $\{z\in \Bbb C: |z-z_0|<\delta\}$, counting multiplicity.
\end{enumerate}
\end{thm}

%%%%%%%%%%%%%%%%%%%%%%%%%%%%%

%%%%%%%%%%%%%%%%%%%%%%%%%%%%%%%%%%%%%%%%%%%%%%
\section{Preliminaries}
%%%%%%%%%%%%%%%%%%%%%%%%%%%%%%%%%%%%%%%%%%%%%%%%%%

The use of P\'{o}lya-Hurwitz partial fraction method is employed for the classical Hankel transforms by Cho et al. in \cite{CCP}. The present work can be also viewed as a $q$-counterpart of the results of \cite{CCP}. Herewith we outline Cho et al.'s work briefly. In \cite{CCP}
Cho et al. extended the work of P\'{o}lya \cite{GP} concerning Fourier
cosine and sine transforms to the Hankel transform
\begin{equation}\label{eq:2.1}
H_\nu(f)(z) =\int\limits_{0}^{1} f(t) J_\nu(tz) \sqrt{tz}dt,
\end{equation}
where $J_\nu(z)$ is the Bessel function of the first kind of order $\nu >-1$. In \cite{CCP} the authors introduced the Kernel $\Bbb{J}_\nu(z)$ and the normalized Hankel transform $\Bbb{H}_\nu(f)(z)=\int_{0}^{1}t^{\nu+1/2}f(t)\Bbb{J}_\nu(tz)\,dt$, where 
\begin{equation}\label{eq:2.2}
\Bbb{J}_\nu(z)=\frac{2^\nu \Gamma(\nu+1)}{z^\nu}J_\nu(z), \
\Bbb{H}_\nu(f)(z) =\frac{2^\nu \Gamma(\nu+1)}{z^{\nu+1/2}}H_\nu(f)(z),\ z\ne 0.
\end{equation}
Equation (\ref{eq:2.2}) implies that $J_\nu(z),\, \Bbb{J}_\nu(z)$ have the same zeros for $z\ne 0$ and so do the
Hankel transforms $H_\nu(z),\,\Bbb{H}_\nu(z)$. The authors studied the nature and distribution of
zeros of $\Bbb{H}_\nu(z)$ instead of $H_\nu(z)$ as it belongs to the Laguerre-P\'{o}lya class of entire
functions $\Bbb{LP}$.i.e., the class of real entire functions $G(z)$ that have the form
\begin{equation}
\nonumber G(z)= C z^le^{-az^2+bz}\prod\limits_{m=1}^{w}\left(1-\frac{z}{\delta_m}\right)e^{\frac{z}{\delta_m}},\ 0\leq w\leq \infty,
\end{equation}
where $a\geq 0$, $b, C \in \Bbb R$, $l$ is a non-negative integer, and $(\delta_m)_{m \in \Bbb N}$ is a sequence of non-zero real numbers satisfying $\sum_{m=1}^{w}1/\delta_m^2<\infty.$

For $\nu >-1$, $-1<\mu<\nu+2$. It is proved for $f(t)$
defined on $0<t<1$ and under the integrability condition
\begin{equation}\label{eq:2.3}
\begin{cases}
\displaystyle\int_0^1 |f(t)| dt < \infty , & \nu \ge -1/2,\\
\displaystyle\int_0^1 t^{\nu + 1/2} |f(t)| dt < \infty,  & -1 < \nu < -1/2
\end{cases}
\end{equation}
that the partial fraction expansion
\begin{equation}\label{eq:2.4}
\frac{\Bbb{H}_\nu(f)(z)}{z \Bbb{J}_\mu(z)} = \frac{\Bbb{H}_\nu(f)(0)}{z}
+ \sum_{m=1}^\infty \frac{\Bbb{H}_\nu(f)(j_{\mu,m})}{j_{\mu,m} \Bbb{J}_\mu'(j_{\mu,m})} 
\left( \frac{1}{z - j_{\mu,m}} + \frac{1}{z + j_{\mu,m}} \right), z \in \Bbb{D}_\mu,
\end{equation}
holds, where $\Bbb{D}_\mu:=\Bbb C\setminus \{0,\pm j_{\mu,1}, \pm j_{\mu,2},\cdots,\pm j_{\mu,m}\}$ denotes the positive zeros of $\Bbb{J}_\mu(z)$. The series (\ref{eq:2.4}) converges uniformly on every compact subset of $\Bbb{D}_\mu$. Expansion (\ref{eq:2.4}) can be also deduced from the sampling theory associated with the Hankel transform derived by many authors, see e.g.\cite{ZHB,Z}.

Applying Hurwitz-P\'{o}lya's technique on the partial fraction expansion (\ref{eq:2.4}), it was proved
for the positive integrable $f(t)$ that if $\sigma_m=(-1)^{m+1}\Bbb{H}_\nu(f)(j_{\mu,m})$, $m=1,2,\cdots,$ keeps the same sign for all $m$, then
$\Bbb{H}_\nu(f)(z)$ has only infinite, real and simple zeros. Moreover, $\Bbb{H}_\nu(f)(z)$ has only one zero in each of the intervals $\left(j_{\mu,m},\ j_{\mu,m+1}\right)$, $m=1,2,\cdots,$ for $\sigma_m>0$
and it has only one zero in each of the intervals
$\left(0, j_{\mu,1}\right)$, $\left(j_{\mu,m},\ j_{\mu,m+1}\right)$, $m=1,2,\cdots,$ for $\sigma_m<0$.

In the special case $\mu=\nu$, they employed a version of Sturm's comparison theorems to find the conditions on $f(t)$ to keep the same sign of $\sigma_m$.

For $\nu>-1$ and $f(t)$ is positive on $0<t<1$ satisfying (\ref{eq:2.3}) and the assumptions:
\begin{enumerate}
\item [\textup{(i)}] $f(t)$ is increasing for $0<t<1$ when $|\nu|\leq 1/2$, and it does not belong to the exceptional case, the class of all step functions on $[0,1]$ having finitely many jump discontinuities at rational points when $|\nu|=1/2$.	
\item [\textup{(ii)}] If $t^{3/2-3}|\nu|f(t)$ is increasing for $0<t<1$ when $|\nu|> 1/2$.
\end{enumerate}

$\Bbb{H}_\nu(f)(z)\in \Bbb{LP}$ with infinite, real and simple zeros. Moreover, $\Bbb{H}_\nu(f)(z)$ has only one positive zero in each interval
\[\left(j_{\nu,m},\ j_{\nu,m+1}\right),\ m=1,2,\cdots. \]

%%%%%%%%%%%%%%%%%%%%%%%%%%%%%%%%%%%%%%%%%%%%%%
\section{Zeros of $H_q^\nu(f;z)$ and $F_q^\nu(f;z)$}
%%%%%%%%%%%%%%%%%%%%%%%%%%%%%%%%%%%%%%%%%%%%%%%%%%

In this section we will derive $q$-analog of P\'{o}lya's results for the zeros of the Hahn-Exton $q$-Bessel function using P\'{o}lya-Hurwitz approach of partial fraction expansion. For this purpose, we implement the
$q$-sampling theorems stated above. We start our derivations
with the following preliminary lemma from \cite{AnSh1}.
\begin{lem}\label{lem:3.1}
Let $x_1<x_2<\cdots<x_n,\; n\ge2,$ be real numbers. Hence, the rational function 
\begin{equation}\label{eq:3.1}
r(x)=\sum_{k=1}^{n}\frac{\alpha_k}{x-x_k}
\end{equation}
has exactly $n-1$ real simple zeros, one in each interval $\left(x_m,x_{m+1}\right),\; m=1,2,\cdots, n-1,$ provided that
\begin{equation}\label{eq:3.2}
\textup{\sign}\,(\alpha_m\,\alpha_{m+1})=1,\quad m=1,2,\cdots, n-1.
\end{equation}
\end{lem}

The next two theorems are the $q$-analogs of P\'{o}lya's results, concerning the zeros of $H_q^\nu(f;z)$ and $F_q^\nu(f;z)$.
\begin{thm}\label{thm:3.2}
The zeros of the entire function $H_q^\nu(f;z)$ are infinite, real, and simple, provided that $f \in L^2_q(0,1)$ and its $q$-Fourier coefficients with respect to the system $\left\{\left(qj_{k\nu}t\right)^{\frac{1}{2}}J_\nu\left(qj_{k\nu}t;q^2\right) \right\}_{k=1}^\infty$ alternate. Moreover, the positive zeros of the function $H_q^\nu(f;z)$ are distributed in the intervals
\begin{equation}\label{eq:3.3}
\left(qj_{k\nu}, qj_{(k+1)\nu}\right),\ k=1,2,\cdots,
\end{equation}
one in each interval.
\end{thm}

\begin{proof}
Since $\left\{\left(qj_{k\nu}t\right)^{\frac{1}{2}}J_\nu\left(qj_{k\nu}t;q^2\right) \right\}_{k=1}^\infty$ is a complete orthogonal set in $L^2_q(0,1)$ cf.\cite{AbBu1}, then $f(\cdot)$ has the $q$-Fourier expansion
\begin{equation}\label{eq:3.4}
f(t)= a_1\left(qj_{1\nu}t\right)^{\frac{1}{2}}J_\nu\left(qj_{1\nu}t;q^2\right)+a_2\left(qj_{2\nu}t\right)^{\frac{1}{2}}J_\nu\left(qj_{2\nu}t;q^2\right)+\cdots,
\end{equation}
where
\begin{equation}\label{eq:3.5}
\displaystyle a_k=\frac{\displaystyle\int_{0}^{1}f(t)\left(qj_{k\nu}t\right)^{\frac{1}{2}}J_\nu\left(qj_{k\nu}t;q^2\right)\,d_qt}{\displaystyle\int_{0}^{1}\left(qj_{k\nu}t\right)\left(J_\nu\left(qj_{k\nu}t;q^2\right)\right)^2\,d_qt}=\frac{H_q^\nu\left(f;qj_{k\nu}\right)}{\|\left(qj_{k\nu}t\right)^{\frac{1}{2}}J_\nu\left(qj_{k\nu}t;q^2\right)\|^2},
\end{equation}
and $\|\left(qj_{k\nu}t\right)^{\frac{1}{2}}J_\nu\left(qj_{k\nu}t;q^2\right)\|\ne 0$, $k=1,2,\cdots$. 
From the conditions of the theorem,
\begin{equation}\label{eq:3.6}
\textup{\sign}\left(a_{m}a_{m+1}\right)=-1,\ m=1,2,\cdots,
\end{equation}
which implies that
\begin{equation}\label{eq:3.7}
\sign\left(H_q^\nu\left(f;qj_{k\nu}\right)H_q^\nu\left(f;qj_{(k+1)\nu}\right)\right)=-1, \;k=1,2,\cdots.
\end{equation}

It follows directly that $H_q^\nu(f;z)$ has at least one zero in each of the intervals (\ref{eq:3.3}). Next we prove that they are simple and there is no other zeros. From expansion (\ref{eq:1.18}) above, we obtain the
$q$-partial fraction expansion
\begin{equation}\label{eq:3.8}
\frac{z^{1/2}H_q^\nu(f;z)}{J_\nu(q^{-1}z; q^2)} = \sum_{k=1}^{\infty} \frac{\left(q j_{k\nu}\right)^{\frac{1}{2}}H_q^\nu(f;qj_{k\nu})}{\alpha_k}
\left( \frac{1}{z - qj_{k\nu}} + \frac{1}{z + qj_{k\nu}} \right),
\end{equation}
$z \ne qj_{k\nu}$, $k=1,2,\cdots$. Therefore
\begin{equation}\label{eq:3.9}
\frac{z^{1/2}H_q^\nu(f;z)}{J_\nu(q^{-1}z; q^2)} =\lim\limits_{m\longrightarrow \infty}\xi_m(z),
\end{equation}
where
\begin{equation}\label{eq:3.10}
\xi_m(z) = \sum_{k=-m, k\ne 0}^{m}\frac{A_k}{z - qj_{k\nu}},\quad A_k=A_{-k}=\frac{\left(q j_{k\nu}\right)^{\frac{1}{2}}H_q^\nu(f;qj_{k\nu})}{\alpha_k}
\end{equation}

and $j_{-k\nu}=-j_{k\nu}$, uniformly on compact subsets of $\Bbb C$ that do not contain any of the numbers $\pm qj_{k\nu}$, $k=1,2,\cdots$. The completion of the proof is hanged up on Hurwitz' theorem, Theorem \ref{thm:1.3} above and proving that
\begin{equation}\label{eq:3.11}
\sign\left\{ \left(\frac{H_q^\nu\left(f;qj_{k\nu}\right)}{\alpha_{k}}\right)\,\left(\frac{H_q^\nu\left(f;qj_{(k+1)\nu}\right)}{\alpha_{k+1}}\right)\right\}=1,\ k=1,2,\cdots.
\end{equation}
Indeed, from (\ref{eq:1.19}), we have $\alpha_k=q^{-1}J_\nu'\left(j_{k\nu};q^2\right),\ k=1,2,\cdots$.

Since the ${}_1\phi_1$-series in (\ref{eq:1.2}) yields $1$ for $z=0$. Then, by continuity,
\begin{equation}\label{eq:3.12}
J_\nu\left(z;q^2\right)>0,\ z \in \left(0,j_{1\nu}\right), \nu>-1.
\end{equation}

Therefore $\alpha_1=q^{-1}J_\nu'\left(j_{1\nu};q^2\right)<0$.
By mathematical induction, we conclude that $\sign\left(\alpha_k\right)$ alternates, and (\ref{eq:3.11}) holds true, then Lemma \ref{lem:3.1} guarantees that $\xi_m(z)$ has exactly $2m-1$ real and simple zeros distributed such that each interval $\left(-qj_{(k+1)\nu}, -qj_{k\nu}\right),\left(-qj_{1\nu}, qj_{1\nu}\right), \left(qj_{k\nu}, qj_{(k+1)\nu}\right),\ k=1,2,\cdots,m-1,$ contains exactly one zero. Hence, from Hurwitz' theorem $H_q^\nu(f;z)$ has one simple zero in each of the intervals $\left(-qj_{(k+1)\nu}, -qj_{k\nu}\right),\left(-qj_{1\nu}, qj_{1\nu}\right),\left(qj_{k\nu}, qj_{(k+1)\nu}\right),\ k\in \Bbb N$.
Relation (\ref{eq:3.7}) implies that $H_q^\nu(f;z)$ cannot have zeros at the end-points of $\left(qj_{k\nu}, qj_{(k+1)\nu}\right)$. Moreover, it has no other zeros from Hurwitz' theorem, and the proof is complete.
\qed
\end{proof}

\begin{thm}\label{thm:3.3}
For $\nu \in \left(0,1\right)-\{1/2\}$, suppose that $f(t) \in L^2_q(0,1)$, and that its $q$-Fourier coefficients alternate, i.e.
\begin{equation}\label{eq:3.13}
f(t)= \sum\limits_{k=1}^{\infty}b_k
J_\nu(z_k;q^2) t^{1/2}\left(c^{-1}z_k^{-2\nu}J_\nu(z_kt;q^2)-J_{-\nu}(z_kq^{-\nu}t;q^2)\right)
\end{equation}
and
\begin{equation}\label{eq:3.14}
\textup{\sign}\left(b_{m}b_{m+1}\right)=-1,\ m=1,2,\cdots.
\end{equation}
Then the zeros of the entire function $F_q^\nu(f;z)$ are infinite, real, and simple. Moreover, the positive zeros of the function $F_q^\nu(f;z)$ are distributed in the intervals
\begin{equation}\label{eq:3.15}
\left(z_{k}, z_{k+1}\right),\ k=1,2,\cdots,
\end{equation}
one in each interval.
\end{thm}

\begin{proof}
The system $\left\{J_\nu(z_k;q^2) t^{1/2}\left(c^{-1}z_k^{-2\nu}J_\nu(z_kt;q^2)-J_{-\nu}(z_kq^{-\nu}t;q^2)\right)
 \right\}_{k=1}^\infty$
is a complete orthonormal set in $L^2_q(0,1)$ cf.\cite{AnMS2012}. Therefore, series (\ref{eq:3.13}) converges in the $L^2_q(0,1)$-norm,
\begin{eqnarray}\label{eq:3.16}
\nonumber\displaystyle b_k&=&\frac{\displaystyle\int_{0}^{1}f(t)J_\nu(z_k;q^2) t^{1/2}\left(c^{-1}z_k^{-2\nu}J_\nu(z_kt;q^2)-J_{-\nu}(z_kq^{-\nu}t;q^2)\right)
\,d_qt}{\displaystyle\int_{0}^{1}\left(J_\nu(z_k;q^2) t^{1/2}\left(c^{-1}z_k^{-2\nu}J_\nu(z_kt;q^2)-J_{-\nu}(z_kq^{-\nu}t;q^2)\right)
\right)^2\,d_qt}\\
&=&\frac{-\left\{\frac{q^{-\nu\left(\nu-1\right)-{1/2}}}{1+q}\Gamma_{q^2}(\nu)\Gamma_{q^2}(1-\nu)\right\}^{-1}F_q^\nu\left(f;z_k\right)}{\left\|J_\nu(z_k;q^2) t^{1/2}\left(c^{-1}z_k^{-2\nu}J_\nu(z_kt;q^2)-J_{-\nu}(z_kq^{-\nu}t;q^2)
\right) \right\|^2}.
\end{eqnarray}
Since $\|J_\nu(z_k;q^2) t^{1/2}\left(c^{-1}z_k^{-2\nu}J_\nu(z_kt;q^2)-J_{-\nu}(z_kq^{-\nu}t;q^2)\right)
\|\ne 0$, $k=1,2,\cdots,$ then condition (\ref{eq:3.14}) leads us to
\begin{equation}\label{eq:3.17}
\sign\left(F_q^\nu\left(f;z_k\right)F_q^\nu\left(f;z_{k+1}\right)\right)=-1, \;k=1,2,\cdots.
\end{equation}

Therefore $F_q^\nu(f;z)$ has at least one zero in each of the intervals (\ref{eq:3.15}). Now, we prove that they are simple and there is no other zeros. The $q$-partial fraction expansion
\begin{equation}\label{eq:3.18}
\frac{zF_q^\nu(f;z)}{\left(q^{\nu^2}z^{-\nu}J_\nu(z; q^2)-cz^{\nu}J_{-\nu}(q^{-\nu}z; q^2)\right)} = \sum_{k=1}^{\infty} \frac{z_kF_q^\nu(f;z_k)}{\beta_k}
\left( \frac{1}{z - z_k} + \frac{1}{z + z_k} \right),
\end{equation}
$z \ne z_k$, $k=1,2,\cdots$, can be deduced from (\ref{eq:1.18}). Thus
\begin{equation}\label{eq:3.19}
 \frac{zF_q^\nu(f;z)}{\left(q^{\nu^2}z^{-\nu}J_\nu(z; q^2)-cz^{\nu}J_{-\nu}(q^{-\nu}z; q^2)\right)}=\lim\limits_{m\longrightarrow \infty}\phi_m(z),
\end{equation}
where $\phi_m(\cdot)$ is the sequence of partial sums
\begin{equation}\label{eq:3.20}
\phi_m(z) = \sum_{k=-m, k\ne 0}^{m}\frac{B_k}{z -z_k},\quad B_k=B_{-k}=\frac{z_kF_q^\nu(f;z_k)}{\beta_k}
\end{equation}
and $z_{-k}=-z_k$, uniformly on compact subsets of $\Bbb C$ that do not contain any of the numbers $\pm z_k$, $k=1,2,\cdots$. As in the previous proof, we prove that
\begin{equation}\label{eq:3.21}
\sign\left\{ \left(\frac{F_q^\nu\left(f;z_k\right)}{\beta_{k}}\right)\,\left(\frac{F_q^\nu\left(f;z_{k+1}\right)}{\beta_{k+1}}\right)\right\}=1,\ k=1,2,\cdots.
\end{equation}
Indeed, from (\ref{eq:1.18}), we have $\beta_k=\frac{d}{dz}\left(q^{\nu^2}z^{-\nu}J_\nu(z; q^2)-cz^{\nu}J_{-\nu}(q^{-\nu}z; q^2)\right)_{z=z_k},\ k=1,2,\cdots$.

Since ${}_1\phi_1\left(\begin{array}{c} 0\\ q^{2\nu+2}\end{array}\\|
q^2;0 \right)\textbf{}=1$, then by continuity,
\begin{equation}\label{eq:3.22}
\left(q^{\nu^2}z^{-\nu}J_\nu(z; q^2)-cz^{\nu}J_{-\nu}(q^{-\nu}z; q^2)\right)>0,\ z \in \left(0,z_1\right)
\end{equation}
for $c<\left(q^{2\nu+2};q^2\right)_\infty/\left(q^{-2\nu+2};q^2\right)_\infty$. Therefore $\beta_1<0$.
By mathematical induction, we conclude that $\sign\left(\beta_k\right)$ alternates, and (\ref{eq:3.21}) holds true and 
\begin{equation}\label{eq:3.23}
\left(q^{\nu^2}z^{-\nu}J_\nu(z; q^2)-cz^{\nu}J_{-\nu}(q^{-\nu}z; q^2)\right)<0,\ z \in \left(0,z_1\right)
\end{equation}
for $c>\left(q^{2\nu+2};q^2\right)_\infty/\left(q^{-2\nu+2};q^2\right)_\infty$. Therefore $\beta_1>0$.
By mathematical induction, we conclude that $\sign\left(\beta_k\right)$ 
alternates, and (\ref{eq:3.21}) holds true. Lemma \ref{lem:3.1} guarantees that $\phi_m(z)$ has exactly $2m-1$ real and simple zeros distributed such that each interval $\left(-z_{k+1)}, -z_{k}\right),\left(-z_1, z_1\right), \left(z_{k}, z_{k+1}\right),\ k=1,2,\cdots,m-1,$ contains exactly one zero. Consequently, from Hurwitz' theorem $F_q^\nu(f;z)$ has one simple zero in each of the intervals $\left(-z_{k+1)}, -z_{k}\right),\left(-z_1, z_1\right), \left(z_{k}, z_{k+1}\right),\ k\in \Bbb N$. Relation (\ref{eq:3.17}) implies that $Fq^\nu(f;z)$ cannot have zeros at the end-points of $\left(z_{k}, z_{k+1}\right)$. Moreover, it has no other zeros from Hurwitz' theorem, completing the proof.
\qed
\end{proof}

\begin{rem}\label{rem:3.4}
Condition \textup{(\ref{eq:3.14})} is clearly restrictive. We can relax it as follows. Assume that the coefficients alternate starting from some $m_0>0$. i.e. is replaced by
\begin{equation}\label{eq:3.24}
\textup{\sign}\left(b_{m}b_{m+1}\right)=-1,\ m>m_0.
\end{equation}
Then, it is not hard to see that 
\begin{equation}\label{eq:3.25}
\begin{array}{lcl}
\mathcal{F}_q^\nu(f;z)&=& F_q^\nu(f;z)-\displaystyle\sum_{k=-m_0+1}^{m_0-1}\frac{z_kF_q^\nu(f;z_k)}{z\beta_k}\left(q^{\nu^2}z^{-\nu}J_\nu(z; q^2)-cz^{\nu}J_{-\nu}(q^{-\nu}z; q^2)\right)\frac{1}{z-z_k}\\
&=&\displaystyle\sum_{|k|\geq m_0}\frac{z_kF_q^\nu(f;z_k)}{z\beta_k}\left(q^{\nu^2}z^{-\nu}J_\nu(z; q^2)-cz^{\nu}J_{-\nu}(q^{-\nu}z; q^2)\right)\frac{1}{z-z_k}
\end{array}
\end{equation}
satisfies
\begin{equation}\label{eq:3.26}
\mathcal{F}_q^\nu(f;z_k)=F_q^\nu(f;z_k), \quad|k|\geq m_0.
\end{equation}
Therefore, and following the same steps of the previous proof, we conclude that $\mathcal{F}_q^\nu(f;z)$ has only real simple simple zeros, which are distributed in the intervals
\begin{equation}\label{eq:3.27}
\left(z_k,z_{k+1}\right), \quad|k|\geq m_0,
\end{equation}
one in each interval.
\end{rem}

%%%%%%%%%%%%%%%%%%%%%%%%%%%%%%%%%%%%%%%%%%

\section{Examples and Comparisons}
%%%%%%%%%%%%%%%%%%%%%%%%%%%%%%%%%%%%%%%%%%

\begin{exa}
Consider the $q$-Hankel transform
\begin{equation}\label{eq:4.1}
 H_q^1(t^{3/2};z):=\int\limits_{0}^{1}t^{3/2}\left(tz\right)^{1/2}J_1\left(tz;q^2\right)\,d_qt=\left(1-q\right)z^{-1/2}J_2\left(z;q^2\right)= z^{1/2}V_{1, t^2}(z)\end{equation}
One can check that $t^{3/2}=f(t) \in L^2_q(0,1)$. Moreover, according to \textup{\cite{AnM06}},
\begin{equation}
\nonumber b_{1,t^2}:=1,\  B_{1,t^2}:=\frac{1}{(1-q^2)(1-q^6)}.
\end{equation}

Therefore, and noting that $g(t)=t^2 \in L^1_q(0,1)$, the zeros of $H_q^1(t^{3/2};z)$ are real, infinite and simple and they are distributed in the intervals,
\begin{equation}\label{eq:4.2}
\left(\mu_k, \mu_{k+1}\right)=\left(q^{-k+1/2}\sqrt{(1-q^2)(1-q^6)}, q^{-k-1/2}\sqrt{(1-q^2)(1-q^6)}\right),\ k=1,2,\cdots,
\end{equation}
one in each interval, provided that 
\begin{equation}\label{eq:4.3}
q^{-1}(1-q)\frac{b_{1,t^2}}{B_{1,t^2}}>1.
\end{equation}
Condition \textup{(\ref{eq:4.3})} restricts $q$ to lie in the interval $\left(0, 0.35118\right)$. Applying \textup{Theorem \ref{thm:3.2}}, we obtain the distributions of the zeros of \textup{(\ref{eq:4.1})} without the restriction \textup{(\ref{eq:4.3})}. Indeed, we first show that the $q$-Fourier coefficients of $f(t)=t^{3/2}$ with respect to the system $\left\{\left(qj_{k1}t\right)^{1/2}J_1\left(qj_{k1}t;q^2\right) \right\}_{k=1}^\infty$
 alternate. Due to \textup{(\ref{eq:3.4})} and \textup{(\ref{eq:3.5})},
\begin{equation}\label{eq:4.4}
f(t)=t^{3/2}=\sum\limits_{k=1}^{\infty}\left(\frac{\left(1-q\right)\left(qj_{k1}\right)^{-1/2}J_2\left(qj_{k1};q^2\right)}{\|\left(qj_{k1}t\right)^{1/2}J_1\left(qj_{k1}t;q^2\right)\|^2}\right)\left(qj_{k1}t\right)^{1/2}J_1\left(qj_{k1}t;q^2\right).
\end{equation}
Since $\sign\left(\mu_k\right)$ alternates and according to \textup{(\ref{eq:1.10})} the coefficients of the $q$-Fourier expansion \textup{(\ref{eq:4.4})} alternate as $\sign \left(J_2\left(qj_{k1};q^2\right)\right)=(-1)^{k+1}$, $k=1,2,\cdots.$ 

Hence for all $q \in (0,1)$, the positive zeros of \textup{(\ref{eq:4.1})} are infinite, real and simple and they lie in the intervals
\begin{equation}\label{eq:4.5}
\left(qj_{k1},qj_{(k+1)1}\right),\ k=1,2,\cdots,
\end{equation}
one in each interval.

\begin{table}\label{tab:1}
\centering	
\begin{tabular}{|c|c|c|c|c|c|}
\hline $q$&
$k$ & Approx $r_k$ &  $(\mu_{k},\mu_{k+1})$&  $(L_{k},L_{k+1})$ & $(\mu_{k},\mu_{k+1})\cap(L_{k},L_{k+1})$\\\hline
\multirow{3}{.5cm}{0.1}
& 1	& 9.99999 & (3.14642, 31.4642)  & (0.999949,10) & (3.14642, 10)      \\\cline{2-6}
& 2	& 99.99999 & (31.4642, 314.642)  & (10,100)  & (31.4642, 100)    \\\hline
\multirow{3}{.5cm}{0.2}
& 1 & 4.99983 & (2.19082, 10.9541)  & (0.999166, 4.99999) & (2.19082, 4.99999)     \\\cline{2-6}
& 2	& 24.99999 & (10.9541, 54.7705)  & (4,99999, 25) &  (10.9541,25)   	\\\hline
\multirow{3}{.5cm}{0.3}
& 1 & 3.332 & (1.74101, 5.80337)  & (0.995543, 3.3333) & (1.74101, 3.3333)          \\\cline{2-6}
& 2 & 11.111111 & (5.80337, 19.3446)  & (3.3333, 11.1111111) & (5.80337, 11.1111111)      	\\\hline
\end{tabular}
\caption{Separation of the zeros $r_k$ of $H_q^1(t^{3/2};z)$ in the asymptotic intervals for $q \in \left(0, 0.35118\right)$.}
\end{table}

\begin{figure*}\label{fig:1}
\centering
\begin{subfigure}[t]{0.5\textwidth}
\centering
\includegraphics[width=4cm, height=3cm]{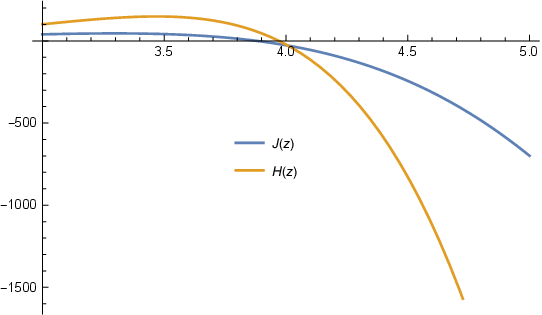}
\caption{q = 0.5}
\end{subfigure}%
~ 
\begin{subfigure}[t]{0.5\textwidth}
\centering
\includegraphics[width=4cm, height=3cm]{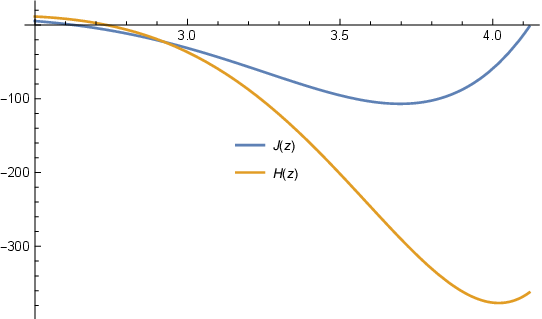}
\caption{q = 0.6}
\end{subfigure}
\begin{subfigure}[t]{0.5\textwidth}
\centering
\includegraphics[width=4cm, height=3cm]{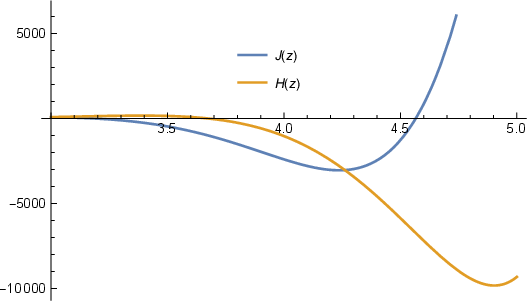}
\caption{q = 0.7}
\end{subfigure}%
\caption{Illustration of the first zero of $H_q^1(t^{3/2};z)$ between the zeros of the function $J_1(q^{-1}z;q^2)$ at different values of $q$.}
\end{figure*}

\begin{table}\label{tab:2}
\centering	
\begin{tabular}{|c|c|c|c|c|}
\hline $q$&
$k$&	 $L_k$&	 Approx $r_k$&  $L_{k+1}$\\\hline
\multirow{3}{.5cm}{0.5}
& 1	& 3.88041 & 3.97127 & 7.65813\\\cline{2-5}
& 2 & 7.65813 & 7.91476 & 15.3279 \\\cline{2-5}
& 3 & 15.3279 & 15.8325 & 30.6595 \\\cline{2-5}
& 4	& 30.6595 & 31.6660 & 61.3209	\\\hline
\multirow{3}{.5cm}{0.6}
& 1 & 2.62559 & 2.73357  & 4.12378\\\cline{2-5}
& 2 & 4.12378 & 4.44658  & 6.92490 \\\cline{2-5}
& 3 & 6.92490 & 7.43061  & 11.5538 \\\cline{2-5}
& 4 & 11.5538 & 12.3890  & 19.2640 \\\hline
\multirow{3}{.5cm}{0.7}
& 1 & 3.18435 & 3.67383  & 4.56578\\\cline{2-5}
& 2 & 4.56578 & 5.25751  & 6.54784 \\\cline{2-5}
& 3 & 6.54784 & 7.52460  & 9.37056\\\cline{2-5}
& 4 & 9.37056 & 10.7585  & 13.3978	\\\hline
\end{tabular}
\caption{Illustration of approximating $r_k$ of $H_q^1(t^{3/2};z)$ in the intervals of approximating zeros of $J_1(q^{-1}z;q^2)$}
\end{table}
Combining \textup{Theorem \ref{thm:3.2}} with the results of \textup{\cite{AnM06}}, we can separate the zeros of $H_q^1(t^{3/2};z)$ as depicted in \textup{Table 1} for $q \in \left(0, 0.35118\right)$.

In \textup{Figure 1}, we demonstrate the first root of $H_q^1(t^{3/2};z)$. We approximate $H_q^1(t^{3/2};z)$ by taking the first $25$ terms of the $q$-integrations and $J_1(q^{-1}z;q^2)$. Notice that the zero of $H_q^1(t^{3/2};z)$ lie in the prescribed interval for various values of $q$. In \textup{Table 2}, we exhibit some numerical results illustrating \textup{Theorem \ref{thm:3.2}}. We can see that all the approximated zeros lie in the asymptotic intervals.

\end{exa}

\begin{exa}
Consider the $q$-transform
\begin{eqnarray}\label{eq:4.6}
F_q^{1/4}(t^{1/4};z)&=&-\frac{q^{-5/16}}{1+q}\int_0^1 t^{3/4}\left(J_{1/4}(tz; q^2)Y_{1/4}(z; q^2)-Y_{1/4}(tz; q^2)J_{1/4}(z; q^2)\right) \;d_qt\\
\nonumber&=& q^{-17/16}\left(1-q\right)^2\frac{\left(q^2;q^2\right)_\infty^2}{\left(q^{1/2};q^2\right)_\infty\left(q^{3/2};q^2\right)_\infty}\left(
\frac{q}{z}J_{1/4}(z;q^2)J_{3/4}(q^{-1/4}z;q^2)\right.\\
\nonumber &{}&\left.+\frac{1}{z}J_{-1/4}(q^{-1/4}z;q^2)J_{-3/4}(q^{-1}z;q^2)-\frac{q^{3/4}}{z^{7/4}}\frac{\left(q^{1/2};q^2\right)_\infty}{\left(q^2;q^2\right)_\infty} J_{-1/4}(q^{-1/4}z;q^2)\right)
 \end{eqnarray}
One can check that $t^{1/4}=f(t) \in L^2_q(0,1)$.

Now, we apply \textup{Theorem \ref{thm:3.3}} on \textup{(\ref{eq:4.6})} to derive the distributions of the it's zeros. Indeed, we first show that the $q$-Fourier coefficients of $f(t)=t^{1/4}$ with respect to the system $\left\{J_{1/4}(z_k;q^2) t^{1/2}\left(c^{-1}z_k^{-1/2}J_{1/4}(z_kt;q^2)-J_{-1/4}(z_kq^{-1/4}t;q^2)\right)
 \right\}_{k=1}^\infty$ alternate. Due to \textup{(\ref{eq:3.13})} and \textup{(\ref{eq:3.16})},
\begin{equation}\label{eq:4.7}
f(t)=t^{1/4}=\sum\limits_{k=1}^{\infty}QF_q^{1/4}\left(t^{1/4};z_k\right)\left\{J_{1/4}(z_k;q^2) t^{1/2}\left(c^{-1}z_k^{-1/2}J_{1/4}(z_kt;q^2)-J_{-1/4}(z_kq^{-1/4}t;q^2)\right)\right\},
\end{equation}
where 
\begin{equation}
\nonumber Q=\frac{-\left\{\frac{q^{-5/16}}{1+q}\Gamma_{q^2}(1/4)\Gamma_{q^2}(3/4)\right\}^{-1}}{\left\|J_{1/4}(z_k;q^2) t^{1/2}\left(c^{-1}z_k^{-1/2}J_{1/4}(z_kt;q^2)-J_{-1/4}(z_kq^{-1/4}t;q^2)
\right) \right\|^2}<0.
\end{equation}
We can see that expansion \textup{(\ref{eq:4.7})} is the $q$-Bessel expansion with respect to the system $\left\{J_{1/4}(z_k;q^2) t^{1/2}\left(c^{-1}z_k^{-1/2}J_{1/4}(z_kt;q^2)-J_{-1/4}(z_kq^{-1/4}t;q^2)\right)
 \right\}_{k=1}^\infty$. Therefore the coefficients of the $q$-Bessel expansion of \textup{(\ref{eq:4.7})} alternate starting from $m_0=3$. 
 
Hence for all $q \in (0,1)$, the zeros of
\begin{equation}\label{eq:4.8}
\begin{array}{lcl}
\mathcal{F}_q^{1/4}(t^{1/4};z)&=& F_q^{1/4}(t^{1/4};z)\\
&{}&-\displaystyle\sum_{k=-2,k\ne 0}^{2}\frac{z_kF_q^{1/4}(t^{1/4};z_k)}{z\beta_k}\left(q^{1/16}z^{-1/4}J_{1/4}(z; q^2)-cz^{1/4}J_{-1/4}(q^{-1/4}z; q^2)\right)\frac{1}{z-z_k}\\
&=&\displaystyle\sum_{|k|\geq 3}^{}\frac{z_kF_q^{1/4}(t^{1/4};z_k)}{z\beta_k}\left(q^{1/16}z^{-1/4}J_{1/4}(z; q^2)-cz^{1/4}J_{-1/4}(q^{-1/4}z; q^2)\right)\frac{1}{z-z_k}
\end{array}
\end{equation}
are infinite, real and simple and they lie in the intervals
\begin{equation}\label{eq:4.9}
\left(z_k,z_{k+1}\right), \quad|k|\geq 3,
\end{equation}
one in each interval.

\begin{figure*}\label{fig:2}
\centering
\begin{subfigure}[t]{0.5\textwidth}
\centering
\includegraphics[width=4cm, height=3cm]{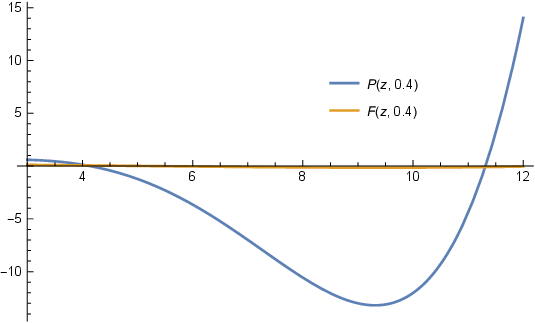}
\caption{q = 0.4}
\end{subfigure}%
~ 
\begin{subfigure}[t]{0.5\textwidth}
\centering
\includegraphics[width=4cm, height=3cm]{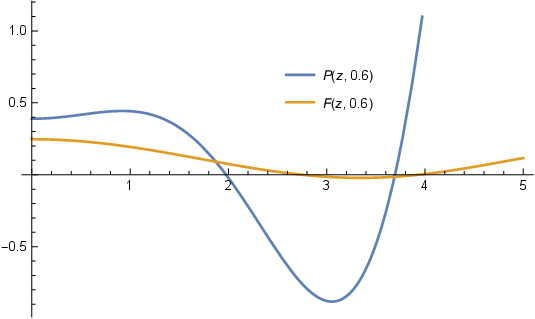}
\caption{q = 0.6}
\end{subfigure}
\begin{subfigure}[t]{0.5\textwidth}
\centering
\includegraphics[width=4cm, height=3cm]{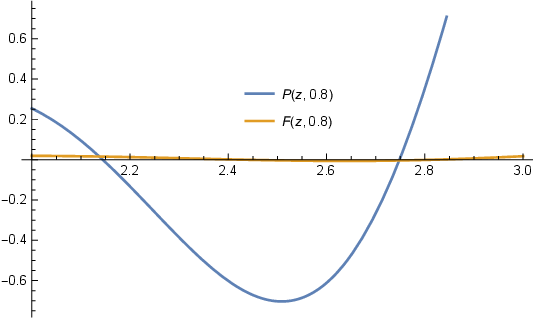}
\caption{q = 0.8}
\end{subfigure}%
\caption{Illustration of the first zero of $F_q^{1/4}(t^{1/4};z)$ between the zeros of the function $q^{1/16}z^{-1/4}J_{1/4}(z; q^2)-z^{1/4}J_{-1/4}(q^{-1/4}z; q^2)$ at different values of $q$.}
\end{figure*}

\begin{table}\label{tab:3}
\centering	
\begin{tabular}{|c|c|c|c|c|}
\hline $q$&
$k$&	 $z_k$&	 Approx $r_k$&  $z_{k+1}$\\\hline
\multirow{3}{.5cm}{0.4}
& 1	& 4.14149 & 5.30337 & 11.3151\\\cline{2-5}
& 2 & 11.3151 & 12.3989 & 29.4680 \\\cline{2-5}
& 3 & 29.4680 & 31.0655 & 75.2827 \\\cline{2-5}
& 4	& 75.2827 & 77.6621 & 190.529	\\\hline
\multirow{3}{.5cm}{0.6}
& 1 & 1.97622 & 2.74773  & 3.69412\\\cline{2-5}
& 2 & 3.69412 & 3.93493  & 6.34216 \\\cline{2-5}
& 3 & 6.34216 & 6.80139  & 10.7773 \\\cline{2-5}
& 4 & 10.7773 & 11.3180  & 18.2025 \\\hline
\multirow{3}{.5cm}{0.8}
& 1 & 0.813412 &         \texttwelveudash & 1.52726\\\cline{2-5}
& 2 & 1.52726  &   \texttwelveudash       & 2.14433 \\\cline{2-5}
& 3 & 2.14433  & 2.42880  & 2.74849\\\cline{2-5}
& 4 & 2.74849  & 2.82145  & 3.46185\\\cline{2-5}	
& 5 & 3.46185  & 3.61722  & 4.35223\\\cline{2-5}
& 6 & 4.35223  & 4.50865  & 5.46686	\\\hline
\end{tabular}
\caption{Illustration of approximating $r_k$ of $F_q^{1/4}(t^{1/4};z)$ in the intervals of approximating zeros of $q^{1/16}z^{-1/4}J_{1/4}(z; q^2)-z^{1/4}J_{-1/4}(q^{-1/4}z; q^2)$.}
\end{table}

In \textup{Figure 2}, we demonstrate the first root of $F_q^{1/4}(t^{1/4};z)$. We approximate $F_q^{1/4}(t^{1/4};z)$ by taking the first $50$ terms of the $q$-integrations and $q^{1/16}z^{-1/4}J_{1/4}(z; q^2)-z^{1/4}J_{-1/4}(q^{-1/4}z; q^2)$. Notice that the zero of $F_q^{1/4}(t^{1/4};z)$ lie in the prescribed interval for various values of $q$. In \textup{Table 3}, we exhibit some numerical results illustrating \textup{Theorem \ref{thm:3.3}}. We can see that all the approximated zeros lie in the asymptotic intervals.
\end{exa}

\noindent{\bf Statements and Declarations}

\noindent{\bf Funding} Both authors declare that there are no financial or non-financial funds for this paper.

\noindent{\bf Data availability}  No data was gathered in this paper.

%%%%%%%%%%%%%%%%%%%%%%%%%%%%%%%%%%%%%%%%%
%%%%%%%%%%%%%%%%%%%%%%%%%%%%%%%%%%%%%%%%%
%%%%%%%%%%%%%%%%%%%%%%%%%%%%%%%%%%%%%%%%%
%%%%%%%%%%%%%%%%%%%%%%%%%%%%%%%%%%%%%%%%%
%%%%%%%%%%%%%%%%%%%%%%%%%%%%%%%%%%%%%%%%%

%\section*{References}
%%%%%%%%%%%%%%%%%%%%%%%%%%%%%%%%%%%%%%%%%%%%%
\bibliographystyle{abbrv}
{\bibliography{shymaaa_refs_01}}

\begin{thebibliography}{10}

\bibitem{AbBu1}
L.~Abreu and J.~Bustoz.
\newblock {\em On the Completeness of Sets of $q$-Bessel Functions
  $J_\nu^{(3)}(x; q)$. In Theory and Applications of Special Functions: A
  Volume Dedicated to Mizan Rahman, M. Ismail and E. Koelink (eds)}, pages
  29--38.
\newblock Springer US, Boston, MA, 2005.

\bibitem{BuCard2}
L.~Abreu, J.~Bustoz, and J.~Cardoso.
\newblock {The roots of the third Jackson q-Bessel functions}.
\newblock {\em Int. J. Math. Math. Sci.}, 67:4241--4248, 2003.

\bibitem{GRR}
G.~Andrews, R.~Askey, and R.~Roy.
\newblock {Special Functions}.
\newblock {\em Cambridge Uni. Press, Cambridge}, 1999.

\bibitem{AnSh1}
M.~Annaby and S.~Elsayed-Abdullah.
\newblock {A $q$-theorem of P\'{o}lya using Hurwitz partial fraction method}.
\newblock {\em Journal of Analysis}, 33:369--385, 2025.

\bibitem{AnHM2012}
M.~Annaby, H.~Hassan, and Z.~Mansour.
\newblock {Sampling theorems associated with singular $q$-Sturm Liouville
  problems}.
\newblock {\em Results. Math.}, 62:121--136, 2012.

\bibitem{AnM06}
M.~Annaby and Z.~Mansour.
\newblock {On the zeros of basic finite Hankel transforms}.
\newblock {\em J. Math. Anal. Appl.}, 323:1091--1103, 2006.

\bibitem{AnM12}
M.~Annaby and Z.~Mansour.
\newblock {\em {$q$-Fractional Calculus and Equations}}.
\newblock Springer, Heidlberg, 2012.

\bibitem{AnMS2012}
M.~Annaby, Z.~Mansour, and I.~Soliman.
\newblock {q-Titchmarsh-Weyl theory: series expansion}.
\newblock {\em Nagoya Mathematical Journal.}, 205:67--118, 2012.

\bibitem{CCP}
Y.-K. Cho, S.-Y. Chung, and Y.~W. Park.
\newblock {Partial fraction expansions and zeros of Hankel transforms}.
\newblock {\em arXiv:2311.04385}, 2023.

\bibitem{H1}
A.~Hurwitz.
\newblock {\"{U}ber die Nullstellen der Bessel'schen Funktion}.
\newblock {\em Math. Ann.}, 33:246--266, 1889.

\bibitem{H2}
A.~Hurwitz.
\newblock {\"{U}ber die Wurzeln einiger transzendenten Gleichungen}.
\newblock {\em Mitteilugen der Mathematischen Gesellschaft in Hamberg},
  2:25--31, 1890.

\bibitem{HR}
H.~Koelink and R.~Swarttouw.
\newblock {On the zeros of the Hahn-Exton q-Bessel functions and associated
  q-Lommel polynomials}.
\newblock {\em J. Math. Anal. Appl.}, 186:690--710, 1994.

\bibitem{GP}
G.~P\'{o}lya.
\newblock {\"{U}ber die Nullstellengewisser ganzer Funktionen}.
\newblock {\em Math. Z}, 2:352--383, 1918.

\bibitem{Sch}
J.~L. Schiff.
\newblock {Normal Families}.
\newblock {\em Springer, New York}, 1993.

\bibitem{Titch}
E.~C. Titchmarsh.
\newblock {The theory of Funtions}.
\newblock {\em The Clarendon Press, 2nd Edition, Oxford}, 1939.

\bibitem{ZHB}
A.~Zayed, G.~Hinsen, and P.~L. Butzer.
\newblock {On Lagrange interpolation and Kramer-type sampling theorems
  associated with Strum-Liouville problems }.
\newblock {\em SIAM J. App. Math}, 50:893--909, 1990.

\bibitem{Z}
A.~I. Zayed.
\newblock {On Kramer's sampling theorem associated with general Strum-Liouville
  boundary-value problems and Lagrange interpolation}.
\newblock {\em SIAM J. App. Math}, 51:575--604, 1991.

\end{thebibliography}

%\bibliographystyle{elsarticle-num}
%\bibliographystyle{spmpsci}
%\bibliography{shymaaa_refs_01}
%%%%%%%%%%%%%%%%%%%%%%%%%%%%%%%%%%%%%%%%%%%%%
%\bibliography{shymaaa_refs_01}
%%%%%%%%%%%%%%%%%%%%%%%%%%%%%%%%%%%%%%%%%
%%%%%%%%%%%%%%%%%%%%%%%%%%%%%%%%%%%%%%%%%
\end{document}